\documentclass[11pt,twoside]{amsart}

\usepackage[
  paper=a4paper,
  headsep=15pt,text={155mm,230mm},centering
]{geometry}

\usepackage[bookmarks]{hyperref}
\usepackage{amssymb,amsxtra}
\usepackage{graphicx,xcolor}
\usepackage{float,array,calc,booktabs,stackrel,url}
\usepackage{enumitem}\setlist[enumerate,1]{font=\upshape}
\usepackage{mathtools}
\usepackage{tikz}
\usetikzlibrary{
  cd,
  calc,
  positioning,
  arrows,
  decorations.pathreplacing,
  decorations.markings,
}
\usepackage[mode=buildnew]{standalone}

\iftrue
    
    \makeatletter
    \def\@settitle{%
      \vspace*{-10pt}
      \begin{flushleft}%
        \LARGE\bfseries
        \strut\@title\strut
      \end{flushleft}%
    }
    \def\@setauthors{%
      \begingroup
      \def\thanks{\protect\thanks@warning}%
      \trivlist
      \raggedright
      \large \@topsep27\p@\relax
      \advance\@topsep by -\baselineskip
    \item\relax
      \author@andify\authors
      \def\\{\protect\linebreak}%
      \authors
      \ifx\@empty\contribs
      \else
      ,\penalty-3 \space \@setcontribs
      \@closetoccontribs
      \fi
      \normalfont
      \endtrivlist
      \endgroup
    }
    \def\@setaddresses{\par
      \nobreak \begingroup
      \small\raggedright
      \def\author##1{\nobreak\addvspace\smallskipamount}%
      \def\\{\unskip, \ignorespaces}%
      \interlinepenalty\@M
      \def\address##1##2{\begingroup
        \par\addvspace\bigskipamount\noindent
        \@ifnotempty{##1}{(\ignorespaces##1\unskip) }%
        {\ignorespaces##2}\par\endgroup}%
      \def\curraddr##1##2{\begingroup
        \@ifnotempty{##2}{\nobreak\noindent\curraddrname
          \@ifnotempty{##1}{, \ignorespaces##1\unskip}\/:\space
          ##2\par}\endgroup}%
      \def\email##1##2{\begingroup
        \@ifnotempty{##2}{\nobreak\noindent E-mail address%
          \@ifnotempty{##1}{, \ignorespaces##1\unskip}\/:\space
          \ttfamily##2\par}\endgroup}%
      \def\urladdr##1##2{\begingroup
        \def~{\char`\~}%
        \@ifnotempty{##2}{\nobreak\noindent\urladdrname
          \@ifnotempty{##1}{, \ignorespaces##1\unskip}\/:\space
          \ttfamily##2\par}\endgroup}%
      \addresses
      \endgroup
      \global\let\addresses=\@empty
    }
    \def\@setabstracta{%
      \ifvoid\abstractbox
      \else
      \skip@17pt \advance\skip@-\lastskip
      \advance\skip@-\baselineskip \vskip\skip@
      \box\abstractbox
      \prevdepth\z@ 
      \vskip-28pt
      \fi
    }
    \renewenvironment{abstract}{%
      \ifx\maketitle\relax
      \ClassWarning{\@classname}{Abstract should precede
        \protect\maketitle\space in AMS document classes; reported}%
      \fi
      \global\setbox\abstractbox=\vtop \bgroup
      \normalfont\small
      \list{}{\labelwidth\z@
        \leftmargin0pc \rightmargin\leftmargin
        \listparindent\normalparindent \itemindent\z@
        \parsep\z@ \@plus\p@
        
      }%
    \item[\hskip\labelsep\bfseries\abstractname.]%
    }{%
      \endlist\egroup
      \ifx\@setabstract\relax \@setabstracta \fi
    }

    \def\ps@headings{\ps@empty
      \def\@evenhead{%
        \setTrue{runhead}%
        \normalfont\scriptsize
        \rlap{\thepage}\hfill
        \def\thanks{\protect\thanks@warning}%
        \leftmark{}{}}%
      \def\@oddhead{%
        \setTrue{runhead}%
        \normalfont\scriptsize
        \def\thanks{\protect\thanks@warning}%
        \rightmark{}{}\hfill \llap{\thepage}}%
      \let\@mkboth\markboth
    }\ps@headings

    \def\section{\@startsection{section}{1}%
      \z@{-1.4\linespacing\@plus-.5\linespacing}{.8\linespacing}%
      {\normalfont\bfseries\Large}}
    \def\subsection{\@startsection{subsection}{2}%
      \z@{-.8\linespacing\@plus-.3\linespacing}{.5\linespacing\@plus.2\linespacing}%
      {\normalfont\bfseries\large}}
    \def\subsubsection{\@startsection{subsubsection}{3}%
      \z@{.7\linespacing\@plus.2\linespacing}{-1.5ex}%
      {\normalfont\bfseries}}
    \def\paragraph{\@startsection{paragraph}{4}%
      \z@{.7\linespacing\@plus.2\linespacing}{-1.5ex}%
      {\normalfont\itshape}}
    \def\@secnumfont{\bfseries}

    \renewcommand\contentsnamefont{\bfseries}
    \def\@starttoc#1#2{\begingroup
      \setTrue{#1}%
      \par\removelastskip\vskip\z@skip
      \@startsection{}\@M\z@{\linespacing\@plus\linespacing}%
      {.5\linespacing}{
        \contentsnamefont}{#2}%
      \ifx\contentsname#2%
      \else \addcontentsline{toc}{section}{#2}\fi
      \makeatletter
      \@input{\jobname.#1}%
      \if@filesw
      \@xp\newwrite\csname tf@#1\endcsname
      \immediate\@xp\openout\csname tf@#1\endcsname \jobname.#1\relax
      \fi
      \global\@nobreakfalse \endgroup
      \addvspace{32\p@\@plus14\p@}%
      \let\tableofcontents\rela\x
    }
    \def\contentsname{Contents}
    \def\l@section{\@tocline{2}{.5ex}{0mm}{5pc}{}}
    \def\l@subsection{\@tocline{2}{0pt}{2em}{5pc}{}}
    \makeatother

\fi

\def\to{\mathchoice{\longrightarrow}{\rightarrow}{\rightarrow}{\rightarrow}}
\makeatletter
\newcommand{\shortxra}[2][]{\ext@arrow 0359\rightarrowfill@{#1}{#2}}
\def\longrightarrowfill@{\arrowfill@\relbar\relbar\longrightarrow}
\newcommand{\longxra}[2][]{\ext@arrow 0359\longrightarrowfill@{#1}{#2}}
\renewcommand{\xrightarrow}[2][]{\mathchoice{\longxra[#1]{#2}}%
  {\shortxra[#1]{#2}}{\shortxra[#1]{#2}}{\shortxra[#1]{#2}}}
\makeatother

\makeatletter
\def\addtagsub#1{\let\oldtf=\tagform@\def\tagform@##1{\oldtf{##1}\hbox{$_{#1}$}}}
\makeatother

\makeatletter
\def\Nopagebreak{\@nobreaktrue\nopagebreak}
\makeatother

\makeatletter
\newtheoremstyle{theorem-giventitle}
        {}{}              
        {\itshape}                      
        {}                              
        {\bfseries}                     
        {.}                             
        {\thm@headsep}                             
        {\thmnote{\bfseries#3}}
\newtheoremstyle{theorem-givenlabel}
        {}{}              
        {\itshape}                      
        {}                              
        {\bfseries}                     
        {.}                             
        {\thm@headsep}                             
        {\thmname{#1}~\thmnumber{#3}\setcurrentlabel{#3}}
\newtheoremstyle{definition-giventitle}
        {}{}              
        {}                      
        {}                              
        {\bfseries}                     
        {.}                             
        {\thm@headsep}                             
        {\thmnote{\bfseries#3}}
\def\setcurrentlabel#1{\gdef\@currentlabel{#1}}
\makeatother

\newtheorem{theorem}{Theorem}[section]
\newtheorem{theoremalpha}{Theorem}

\newtheorem{lemma}[theorem]{Lemma}

\theoremstyle{definition}
\newtheorem{definition}[theorem]{Definition}

\newtheorem{remark}[theorem]{Remark}

\newtheorem*{case2'}{Case 2$'$}

\theoremstyle{theorem-giventitle}
\newtheorem{theorem-named}{}
\theoremstyle{theorem-givenlabel}
\newtheorem{theorem-labeled}{Theorem}

\theoremstyle{definition-giventitle}
\newtheorem{definition-named}{}
\newtheorem{conjecture-named}{}
\newtheorem{case-named}{}

\numberwithin{equation}{section}

\def\Z{\mathbb{Z}}
\def\R{\mathbb{R}}
\def\Q{\mathbb{Q}}
\def\C{\mathbb{C}}
\def\P{\mathbb{P}}

\def\cN{\mathcal{N}}

\def\tilde{\widetilde}

\def\sm{\smallsetminus}

\def\Im{\operatorname{Im}}

\DeclareMathOperator\Hom{Hom}

\DeclareMathOperator\sign{sign}

\def\rank{\operatorname{rank}}

\def\rhot{\rho^{(2)}}
\def\lt{L^2}

\def\L{\Lambda}
\def\sn{\mathrm{sn}}
\def\dsn{\mathrm{dsn}}
\def\wdsn{\mathrm{wdsn}}
\def\gds{g_{ds}}
\def\gdsx{g_{ds}^X}

\makeatletter

\begin{document}

\title[Double slice genus]{Doubly slicing knots and embedding 3-manifolds in 4-manifolds}

\author{Se-Goo Kim}
\address{Department of Mathematics\\
Kyung Hee University\\
Seoul 02447\\
Republic of Korea
}
\email{sgkim@khu.ac.kr}

\author{Taehee Kim}
\address{Department of Mathematics\\
  Konkuk University\\
  Seoul 05029\\
  Republic of Korea
}
\email{tkim@konkuk.ac.kr}

\thanks{The second-named author was supported by the National Research Foundation of Korea(NRF) grant funded by the Korea government(MSIT) (No.~RS-2023-NR076429).}

\def\subjclassname{\textup{2010} Mathematics Subject Classification}
\expandafter\let\csname subjclassname@1991\endcsname=\subjclassname
\expandafter\let\csname subjclassname@2000\endcsname=\subjclassname

\begin{abstract} 
For a knot $K$ in the 3--sphere and a simply connected closed 4-manifold $X$, we define the $X$--double slice genus of  $K$, extending the notion from the case when $X$ is the 4--sphere. We show that for each integer $n$, there exists an algebraically doubly slice and ribbon knot $K$ whose $X$--double slice genus is greater than $n$. 
Our arguments use new $\lt$--signature obstructions to embedding closed 3-manifolds with infinite cyclic first homology into closed 4-manifolds with infinite cyclic fundamental group, in a way that preserves first homology. 

We also extend the concept of the superslice genus of a knot $K$ to simply connected 4-manifolds and show that there exist doubly slice knots whose generalized superslice genera are arbitrarily large. 

Furthermore, we define the double stabilizing number of a knot, extending the stabilizing number introduced by Conway and Nagel, and show that this invariant can also be arbitrarily large. 
\end{abstract}
\maketitle

\section{Introduction}\label{section:introduction}
In the study of topology in three and four dimensions, many fundamental questions arise concerning the interactions between knots in the 3--sphere $S^3$ and surfaces in the 4--sphere~$S^4$. 
Among these, questions involving slice knots are of primary importance; a knot is \emph{slice} if it bounds a topologically locally flat disk in the 4-ball $D^4$. Equivalently, a knot is slice if it arises as the transverse intersection of a topologically locally flat 2--sphere and the standard~$S^3$ in~$S^4$. 

An oriented, closed, locally flatly embedded surface in $S^4$ is \emph{unknotted} if it bounds a $3$--dimensional handlebody with a locally flat embedding in $S^4$.
Fox \cite{Fox:1962-2} asked which slice knots can arise as the transverse intersection of an unknotted 2--sphere and the standard~$S^3$ in~$S^4$. Such a knot is called \emph{doubly slice}. 

\subsubsection*{Conventions} In this paper, all manifolds are oriented, compact, connected, and topological, and all embeddings are locally flat unless otherwise specified. 

\bigskip
To measure how close a knot is to being slice or doubly slice, we use the 4--\emph{genus} (also called the \emph{slice genus}) or the \emph{double slice genus}. 
The 4--\emph{genus} of a knot $K$ in~$S^3$, denoted by $g_4(K)$, is the minimal genus of a properly embedded surface in $D^4$ with boundary~$K$. 
The \emph{double slice genus} of $K$ \cite[Section~5]{Livingston-Meier:2015-1} is defined as
\[
\gds(K)=\text{min}\{\,\text{genus}(F) \mid F \text{ is an unknotted surface in } S^4,\, F\cap S^3=K \}.
\] 
A knot $K$ is doubly slice if and only if $\gds(K)=0$, and 
\[
2g_4(K)\le \gds(K)\le 2g_3(K)
\] 
(see \cite[Section~5]{Livingston-Meier:2015-1}), where $g_3(K)$ denotes the Seifert genus of $K$.

The notion of the 4--genus of a knot $K$ extends to the genus of $K$ in a 4-manifold~$X$ with boundary $S^3$. The \emph{$X$--genus} of $K$, denoted by~$g_X(K)$, is defined as the minimal genus of a properly embedded surface in $X$ with boundary~$K$. Kasprowski, Powell, Ray, and Teichner \cite[Corollary~1.15]{Kasprowski-Powell-Ray-Teichner:2024-01} showed that $g_X(K)=0$ for every knot $K$ if~$X$ is a simply connected 4-manifold not homeomorphic to the complement of a 4--ball in $S^4$, $\C\P^2$, or~$*\C\P^2$. 

A further refinement involves fixing a homology class represented by an embedded surface bounding $K$. A knot $K$ in $S^3$ is {\it $H$--slice} in a 4-manifold $X$ with boundary $S^3$ if it bounds a null-homologous disk properly embedded in~$X$. 
(A properly embedded surface $F$ in a 4-manifold $W$ is \emph{null-homologous} if $[F,\partial F]=0\in H_2(W, \partial W)$. Henceforth, (co)homology groups are understood with integer coefficients unless mentioned otherwise.) For instance, every properly embedded surface in $D^4$ is null-homologous. Klug and Ruppik \cite[Theorem~4.4]{Klug-Ruppik:2021-1} showed that for each 4-manifold $X$ with boundary $S^3$, there exists a knot that is not $H$--slice in $X$.

\subsubsection*{Double slice genus}

In this paper, we extend the notion of the double slice genus of $K$ to the double slice genus of $K$ in (pairs of) simply connected 4-manifolds. 
A surface~$S$ (properly) embedded in a 4-manifold $X$ is called a {\it $\Z$--surface} if $\pi_1(X\sm S)\cong\Z$, and a $\Z$--surface~$S$ is called a {\it $\Z$--sphere} if $S$ is homeomorphic to~$S^2$. 
For instance, an unknotted surface in $S^4$ is a $\Z$--surface in $S^4$. 

We remark that a $\Z$--surface in a simply connected 4-manifold is null-homologous \cite[Lemma~5.1]{Conway-Powell:2023-1}. In this paper, every 4-manifold will be assumed to be simply connected, and hence every $\Z$--surface will be null-homologous.

Let $X_i$ be simply connected 4-manifolds with boundary $S^3$, $i=1,2$. 
We say that a knot~$K$ is \emph{$(X_1,X_2)$--doubly slice} if there exist properly embedded disks $D_i\subset X_i$ with $\partial D_i=K$, $i=1,2$, such that $D_1\cup_K D_2$ is a $\Z$--sphere in $X_1\cup_{S^3} X_2$.
The \emph{$(X_1,X_2)$--double slice genus} of~$K$ is defined as 
\[
g_{X_1,X_2}(K)=\text{min}\{\,\text{genus}(F_1\cup_K F_2) \mid F_i\subset X_i \text{ and } F_1\cup_K F_2 \text{ is a } \Z\text{--surface in }X_1\cup_{S^3} X_2\},
\] 
where each $F_i$ runs over all properly embedded surfaces in $X_i$ with boundary $K$, $i=1,2$. 
Clearly, $g_{X_1,X_2}(K)=0$ if and only if $K$ is $(X_1,X_2)$--doubly slice.

Let $X$ be a simply connected \emph{closed} 4-manifold. Notice that if $X=X_1\cup_{S^3}X_2$ for 4-manifolds $X_i$ with boundary $S^3$, $i=1, 2$, then each $X_i$ is simply connected. We say that a knot $K$ is \emph{$X$--doubly slice} if there exist 4-manifolds $X_i$ with boundary $S^3$, $i=1,2$, such that $X=X_1\cup_{S^3} X_2$ and $K$ is $(X_1,X_2)$--doubly slice. 
The \emph{$X$--double slice genus of $K$} is defined as
\begin{align*}
\gdsx(K)=\text{min}\{g_{X_1,X_2}(K)\mid\, & X_i\text{ are 4-manifolds with boundary }S^3, i=1,2, \\
& \text{ such that } X=X_1\cup_{S^3}X_2\}.
\end{align*}
One can readily see that a knot $K$ is $X$--doubly slice if and only if $\gdsx(K)=0$. We have $\gds^{S^4}(K)\le \gds(K)$ since every unknotted surface in $S^4$ is a $\Z$--surface. We will also show that $\gdsx(K)\le \gds^{S^4}(K)$ for every simply connected closed 4-manifold $X$ in Theorem~\ref{theorem:inequalities-double-slice-genus}.

\begin{remark}\label{remark:unknotting-conjecture}
\begin{enumerate}
\item If a surface $F$ is unknotted in $S^4$, then $\pi_1(S^4\sm F)\cong \Z$. 
The \emph{unknotting conjecture} asserts the converse, and the unknotting conjecture holds when $\text{genus}(F)\ne 1,2$ (see \cite[Theorem~11.7A]{Freedman-Quinn:1990-1} for $\text{genus}(F)=0$, \cite[Theorem~1.1]{Conway-Powell:2023-1} for $\text{genus}(F)\ge 3$, and \cite[Seciton~1.1]{Conway-Powell:2023-1} for related discussions). 
It follows that a knot is doubly slice if and only if it is $S^4$--doubly slice in the above sense. 
Moreover, $\gds^{S^4}(K)\le \gds(K)$, and if the unknotting conjecture is true, then $\gds^{S^4}(K)= \gds(K)$.
\item In the above definitions of $g_{X_1,X_2}(K)$ and $\gdsx(K)$, the surfaces $D_i$ and $F_i$ in $X_i$ are  null-homologous, $i=1, 2$, since $D_1\cup_K D_2$ and $F_1\cup_K F_2$ are $\Z$--surfaces in a simply connected 4-manifold and hence null-homologous (see \cite[Lemma~5.1]{Conway-Powell:2023-1}).
\end{enumerate}
\end{remark}

We will give lower bounds on $\gdsx(K)$ and show that it can be arbitrarily large by using new obstructions to embedding 3-manifolds with infinite cyclic first homology in 4-manifolds with infinite cyclic fundamental group (Theorem~\ref{theorem:H1-embedding}), which may be of independent interest.

\subsubsection*{Obstructions to embedding 3-manifolds in 4-manifolds}

Let $M$ be a closed 3-manifold and $W$ be a closed 4-manifold. We say that an embedding $M\hookrightarrow W$ is an \emph{$H_1$--embedding} if it induces an isomorphism $H_1(M)\xrightarrow{\cong} H_1(W)$. Let $M(K)$ denote the zero-framed surgery on a knot $K$ in $S^3$. We will show that $H_1$--embedding is closely related to the $X$--double slice genus, and using it we will give lower bounds on the $X$--double slice genus. In Theorem~\ref{theorem:doubly-slice-and-H_1-embedding}, we also show that a knot $K$ is doubly slice if and only if there exists an $H_1$--embedding of~$M(K)$ in~$S^1\times S^3$.

For a closed 3-manifold $M$ and a group homomorphism $\phi\colon \pi_1(M)\to G$,  the \emph{Cheeger--Gromov--von Neumann $\rhot$-invariant}, denoted by $\rhot(M,\phi)$, is defined as a $\lt$--signature defect of a 4-manifold over which $\phi$ extends (see Section~\ref{section:l2-signatures} for more about $\rhot$-invariants). 

Using $\rhot$-invariants  we obtain the following $H_1$--embedding obstructions. 
In the following theorem, a group~$G$ will be assumed to be \emph{amenable} and in \emph{Strebel's class $D(R)$} with $R=\Q$ or $\Z_p$ for a prime $p$, and refer to Section~\ref{section:embedding} for the definitions and properties of these terms.
For a space~$X$, let $b_i(X)=\rank_\Z H_i(X)$.

\begin{theoremalpha}\label{theorem:H1-embedding}
Let $M$ be a closed 3-manifold with $H_1(M)\cong \Z=\langle t\rangle$ and $W$ be a closed 4-manifold with $\pi_1(W)\cong \Z$. If $f\colon M\to W$ is an $H_1$--embedding, then there exist 4-manifolds~$W_1$ and~$W_2$ for which the following hold: let $r=b_2(W)$ and $\L=\Q[t^{\pm 1}]$.
\begin{enumerate}
    	\item $f(M)=\partial W_1=\partial W_2$ and $W=W_1\cup_{f(M)} W_2$.
	\item For $i=1,2$, let $f_i\colon M\to W_i$ be the inclusion map induced from $f$. Then, $f_i$ induces an isomorphism $(f_i)_*\colon H_1(M)\xrightarrow{\cong} H_1(W_i)$, $i=1, 2$. 
	\item There is an exact sequence
	\[
	\L^r\to H_1(M;\L)\xrightarrow{((f_1)_*,(f_2)_*)} H_1(W_1;\L)\oplus H_1(W_2;\L)\to 0.
	\]
	\item Let $G$ be an amenable group lying in Strebel's class $D(R)$ where $R=\Q$ or $\Z_p$ with $p$ a prime. If $\phi\colon \pi_1(M)\to G$ is a homomorphism that extends over either $W_1$ or $W_2$, then 
	\[
	\left|\rhot(M,\phi)\right|\le 2r.
	\]
	\end{enumerate}
\end{theoremalpha}

We relate Theorem~\ref{theorem:H1-embedding} to $g_{X_1,X_2}(K)$ and $g_X(K)$ in the following theorem. 
\begin{theoremalpha}\label{theorem:genus-to-embedding}
Let $K$ be a knot in $S^3$ and $g$ be a nonnegative integer. 
\begin{enumerate}
\item Let $X_i$ be simply connected 4-manifolds with boundary $S^3$, $i=1,2$, and let $X=X_1\cup_{S^3} X_2$. If $g_{X_1,X_2}(K)\le g$, then there exists a closed 4-manifold~$W$ with $\pi_1(W)\cong \Z$ and $H_2(W)\cong \Z^{2g+b_2(X)}$ such that there is an $H_1$--embedding of $M(K)$ in~$W$. 
\item Let $X$ be a simply connected closed 4-manifold. If $\gdsx(K)\le g$, then there exists a closed 4-manifold $W$ with $\pi_1(W)\cong \Z$ and $H_2(W)=\Z^{2g+b_2(X)}$ such that there is an $H_1$--embedding of $M(K)$ in~$W$. 
\end{enumerate}
\end{theoremalpha}

Theorems \ref{theorem:H1-embedding} and \ref{theorem:genus-to-embedding} yield lower bounds on $\gdsx(K)$. Applying these bounds, we obtain the following theorem, which shows that $\gdsx(K)$ can be arbitrarily large for each simply connected closed 4-manifold $X$.
\begin{theoremalpha}\label{theorem:main-ribbon}
Let $g$ and $n$ be nonnegative integers. Then, there exists an algebraically doubly slice and ribbon knot $K$ in $S^3$ such that there is no $H_1$--embedding of $M(K)$ in any closed 4-manifold $W$ with $\pi_1(W)\cong \Z$ and $b_2(W)\le 2g+n$. 
In particular, such a knot $K$ satisfies $\gdsx(K)> g$ for every simply connected closed 4-manifold $X$ with $b_2(X)=n$.
\end{theoremalpha}

\begin{remark}\label{remark:main-ribbon}
\begin{enumerate}
\item In Theorem~\ref{theorem:main-ribbon}, since $K$ is ribbon, in each decomposition $X=X_1\cup_{S^3}X_2$ for 4-manifolds~$X_i$ with boundary~$S^3$, the knot~$K$ bounds a null-homologous disk in the collar of $S^3$ in each~$X_i$. Therefore~$K$ is $H$--slice in each~$X_i$. 

\item When $X=S^4$, since $\gds^{S^4}(K)\le \gds(K)$, Theorem~\ref{theorem:main-ribbon} reproves \cite[Theorem~1.2]{Chen:2021-1}, which states that for any integer~$g$, there exists an algebraically doubly slice and ribbon knot $K$ such that $\gds(K)> g$. 

\item Orson and Powell \cite[Theorem~1.1]{Orson-Powell:2021-1} showed that $\gds(K) \ge |\sigma_K(\omega)|$ for every $\omega\in S^1$, where $\sigma_K$ denotes the Levine--Tristram signature function of $K$. The knots~$K$ in Theorem~\ref{theorem:main-ribbon} have $\sigma_K(\omega)=0$ for every $\omega\in S^1$ since they are algebraically doubly slice. 
\end{enumerate}
\end{remark}

Let $\Sigma_n(K)$ be the $n$-fold branched cyclic cover of $S^3$ over $K$. 
If $K$ is a doubly slice knot, then for each prime power $n$, the 3-manifold $\Sigma_n(K)$ is a rational homology 3--sphere, and it embeds in $S^4$ since the $n$-fold branched cyclic cover of $S^4$ over an unknotted 2--sphere is homeomorphic to $S^4$ (refer to \cite{Gilmer-Livingston:1983-1}). 
A natural question arises: if $\Sigma_n(K)$ embeds in $S^4$ for every prime power $n$, is the knot $K$ doubly slice?

We answer this question in the negative. 
We note that if $H_1(\Sigma_n(K))=0$ (equivalently, if $\Sigma_n(K)$ is an integral homology 3--sphere) for all prime powers $n$, then Casson--Gordon invariants vanish for $K$ \cite{Casson-Gordon:1986-1}. 
We also note that every integral homology 3--sphere embeds in $S^4$ \cite[Theorem~1.4]{Freedman:1982-1}. 

\begin{theoremalpha}\label{theorem:main-vanishing-CG}
Let $g$ be a nonnegative integer. Then, there exists an algebraically doubly slice knot $K$ satisfying the following:
\begin{enumerate}
	\item There is no $H_1$--embedding of $M(K)$ in a closed 4-manifold $W$ with $\pi_1(W)\cong \Z$ and $b_2(W)\le 2g$. In particular, $\gds(K)>g$ and $K$ is not doubly slice. 
	\item For every prime power $n$, we have $H_1(\Sigma_n(K))=0$. In particular, $\Sigma_n(K)$  embeds in~$S^4$, and Casson--Gordon invariants vanish for $K$.
\end{enumerate}	
\end{theoremalpha}
We remark that the lower bounds on $\gds(K)$ in \cite[Theorem~1.1]{Orson-Powell:2021-1} and \cite{Chen:2021-1} do not apply to the knots $K$ in Theorem~\ref{theorem:main-vanishing-CG} since $K$ are algebraically doubly slice and $H_1(\Sigma_2(K))=0$. 

\subsubsection*{Double stabilizing number}

We define \emph{the double stabilizing number} of a knot and relate it to Theorem~\ref{theorem:H1-embedding}. 
A knot $K$ is \emph{stably $H$--slice} if it is $H$--slice in $D^4\# n S^2\times S^2$, the connected sum of $D^4$ and $n$ copies of $S^2\times S^2$, for some $n\ge 0$. 
This occurs if and only if the Arf invariant of $K$ vanishes \cite{Schneiderman:2010-1}. 
Conway and Nagel \cite{Conway-Nagel:2020-1} defined \emph{the stabilizing number} of a stably $H$--slice knot $K$ as
\[
\sn(K)=\text{min} \{\,n\mid K \text{ is }H\text{--slice in }D^4\# nS^2\times S^2\}.
\]
We note that $\sn(K)=0$ if and only if $K$ is slice in $D^4$ since every slice disk in $D^4$ is null-homologous.

We extend the above notions to doubly slice knots. 
For an integer $n\ge 0$, let 
\[
W_n =D^4\# nS^2\times S^2.
\]  
We say that a knot $K$ is {\it stably doubly slice} if it is $(W_m,W_n)$--doubly slice for some integers~$m,n\ge 0$. It is shown in Theorem~\ref{theorem:stably-doubly-slice-Arf-invariant}(1) that a knot is stably doubly slice if and only if it is stably $H$--slice.

For a stably doubly slice knot $K$, we define {\it the double stabilizing number of $K$} as
\[
\dsn(K)=\text{min}\{\,m+n\mid K \text{ is } (W_m,W_n)\text{--doubly slice}\}. 
\]

In Theorem~\ref{theorem:stably-doubly-slice-Arf-invariant}~(2) we show that for a stably doubly slice knot $K$
\begin{equation*}\label{equation:sn-dsn}
2\sn(K)\le \dsn(K)\le \gds(K).
\end{equation*}
We ask whether these inequalities can be strict.
Using Theorem~\ref{theorem:H1-embedding}, we present the following theorem, which shows that the difference $\dsn(K) - 2\sn(K)$ can be arbitrarily large.

\begin{theoremalpha}\label{theorem:main-dsn}
For each integer $m\ge 0$, there exists an algebraically doubly slice and ribbon knot $K$ (hence $\sn(K)=0$) such that $\dsn(K)> m$. 
\end{theoremalpha}

For a simply connected 4-manifold $X$ with boundary $S^3$ and $x\in H_2(X,\partial X)$, Conway, Orson, and Pencovitch \cite[Definition~1.11]{Conway-Orson-Pencovitch:2025-1} extended the stabilizing number of $K$ to the \emph{$(x,X)$--stabilizing number of $K$}. 
Inspired by this, in Subsection~\ref{subsection:X-double-stabilizing-number}, for a simply connected closed 4-manifold~$X,$ we will define the \emph{$X$--double stabilizing number of $K$}, extending the notion of double stabilizing number, and show that it can be arbitrarily large.

\begin{remark}\label{remark:double-stabilizing-number}
Define the \emph{weak double stabilizing number} of $K$ as 
\[
\wdsn(K)=\text{min}\{\,n\mid K \text{ is } mS^2 \times S^2\text{--doubly slice} \}. 
\]
Clearly, $\wdsn(K)\le \dsn(K)$. If $mS^2 \times S^2 = X_1\cup_{S^3} X_2$ for 4-manifolds $X_i$ with $\partial X_i=S^3$, then $H_2(mS^2 \times S^2)\cong H_2(X_1)\oplus H_2(X_2)$. Using this, along with the proof of Theorem~\ref{theorem:main-dsn}, one can show that Theorem~\ref{theorem:main-dsn} still holds when $\dsn(K)>m$ is replaced by $\wdsn(K)>m$. 
But the authors do not know whether $2\sn(K)\le \wdsn(K)$: when $mS^2 \times S^2 = X_1\cup_{S^3} X_2$, possibly $X_1\ncong D^4\# nS^2\times S^2$ for any $n$. 
For instance, $8 S^2\times S^2\cong X\cup (-X)$ where $X=E_8\sm \text{int} (D^4)$.
\end{remark}

\subsubsection*{Superslice genus} We extend the notions of a \emph{superslice knot} and \emph{the supersilce genus} of a knot to simply connected 4-manifolds. 
A knot $K$ in $S^3$ is \emph{superslice} if there exists a properly embedded disk~$D$ in $D^4$ with $\partial D=K$ such that the double of $D$ is unknotted in~$S^4$. 
For brevity, for a manifold~$M$ with boundary, let $DM$ denote the double of $M$.
The \emph{superslice genus of~$K$}, denoted by $g_s(K)$, is the minimal genus of a properly embedded surface~$F$ in~$D^4$ with $\partial F=K$ such that $DF$ is unknotted in~$S^4$. 
It is known that~$g_s(K)$ equals the $\Z$--slice genus of~$K$, where the \emph{$\Z$--slice genus of $K$} is defined as the minimal genus of a properly embedded surface~$F$ in $D^4$ with $\partial F=K$ such that $\pi_1(D^4\sm F)\cong \Z$. (see \cite[Corollary~1.7]{Feller-Lewark:2024-1} and \cite[Section~1]{Chen:2021-1}). 
For more about the $\Z$--slice genus and $\Z$--surfaces, we refer the reader to \cite{Feller-Lewark:2024-1, Conway-Powell:2023-1}.

Let $X$ be a simply connected 4-manifold with boundary $S^3$. We say that a knot $K$ in $S^3$ is \emph{$X$--superslice} if~$K$ bounds a properly embedded disk $D$ in $X$ with $\partial D=K$ such that the double of $D$ is a $\Z$--sphere in $DX$. 
We also define \emph{the $X$--superslice genus of $K$} as 
\[
g_s^X(K)= \text{min}\{\text{genus}(F)\mid DF \text{ is a }\Z \text{--surface in } DX\},
\]
where $F$ runs over all properly embedded surfaces in $X$ with $\partial F=K$. 
(Then, $F$ is null-homologous in $X$. See Remark~\ref{remark:unknotting-conjecture}(2).)

It is readily seen that $\gds^{DX}(K) \le 2 g_s^X(K)$ and $g_s^{D^4}(K)\le g_s(K)$. We note that $g_s^{D^4}(K)= g_s(K)$ if the unknotting conjecture is true. 
In Theorem~\ref{theorem:inequalities-superslice-genus} we show that $g_s^X(K)\le g_s^{D^4}(K)$.

The following theorem gives lower bounds on $g_s^X(K)$. Recall that $\L=\Q[t^{\pm 1}]$.

\begin{theoremalpha}\label{theorem:lower-bound-superslice-genus}
Let $K$ be a knot in $S^3$ and $X$ be a simply connected 4-manifold with boundary~$S^3$. Then, the minimal number of generators of $H_1(M(K);\L)$ as a $\L$-module is less than or equal to $4g_s^X(K) + 2b_2(X)$.
\end{theoremalpha}

Using Theorem~\ref{theorem:lower-bound-superslice-genus}, we obtain the following theorem, which shows that the difference $2g_s^X(K) - g_{ds}^{DX}(K)$ can be arbitrarily large.

\begin{theoremalpha}\label{theorem:superslice-genus-large}
Let $X$ be a simply connected 4-manifold with boundary $S^3$, and let $g$ be an integer. Then, there exists a doubly slice and ribbon knot $K$ such that $g_s^X(K)>g$ and $\gds^{DX}(K)=0$.
\end{theoremalpha}

\bigskip
This paper is organized as follows. 
In Section~\ref{section:l2-signatures}, we recall $L^2$--signatures and $\rhot$-invariants. 
In Section~\ref{section:embedding} we prove Theorem~\ref{theorem:H1-embedding}.
Theorem~\ref{theorem:genus-to-embedding} is proved in Section~\ref{section:genus-embedding}. 
In Sections~\ref{section:double-silce-genus-in-4-manifolds} and \ref{section:vanishing-CG-invariants}, Theorems~\ref{theorem:main-ribbon} and \ref{theorem:main-vanishing-CG} are proved, respectively. 
Theorem~\ref{theorem:main-dsn} is proved in Section~\ref{section:doubly-stabilizing-number}.
Theorems~\ref{theorem:lower-bound-superslice-genus} and \ref{theorem:superslice-genus-large} are proved in Section~\ref{section:superslice-genus}.
 
\subsection*{Acknowledgments} 
We thank Charles Livingston, James Davis, and Jae Choon Cha for helpful conversations.

\section{Cheeger--Gromov--von Neumann $\rhot$-invariants}\label{section:l2-signatures}
In this section, we recall the $\rhot$-invariants of a 3-manifold and related results.

Let $M$ be a closed 3-manifold and $\phi\colon \pi_1(M)\to G$ a homomorphism for a group $G$. 
In \cite{Cheeger-Gromov:1985-1}, for a given Riemannian metric on $M$, Cheeger and Gromov defined the \emph{Cheeger--Gromov--von Neumann $\rhot$-invariant $\rhot(M,\phi)$} using the signature operators of $M$ and the $G$-cover of $M$. They showed that this invariant is independent of the choice of a metric. 
Later, Chang and Weinberger \cite{Chang-Weinberger:2003-1} provided a topological definition of $\rhot(M,\phi)$ as follows. 
For the homomorphism $\phi$, there exist an injective homomorphism $i\colon G\to G'$ and a 4-manifold~$W$ such that $i\circ\phi\colon \pi_1(M)\to G'$ extends to $\pi_1(W)$. 
Then, $\rhot(M,\phi)$ is defined as 
\begin{equation}\label{equation:rho-invariant}
\rhot(M,\phi)=\sign_{G'}^{(2)}(W)-\sign(W) \in \R,
\end{equation}
where $\sign_{G'}^{(2)}(W)$ is the $\lt$--signature of the hermitian intersection form 
\[
H_2(W;\cN G')\times H_2(W;\cN G')\to \cN G'
\] 
for the group von Neumann algebra $\cN G'$, and $\sign(W)$ is the ordinary signature of~$W$. It is known that the above definition of $\rhot(M,\phi)$ is independent of the choice of the both the map $i\colon G\to G'$ and the 4-manifold $W$.
For more details about $\rhot$-invariants and $\lt$--signatures, see \cite{Cochran-Orr-Teichner:2003-1,Cochran-Teichner:2003-1,Cha:2010-01}. 

Below, we provide a list of properties of $\rhot$-invariants that will be used in this paper. Recall that for a knot $K$, the 3-manifold $M(K)$ denotes the zero-framed surgery on $K$ in $S^3$ and $\sigma_K$ denotes the Levine--Tristram signature function of $K$. 
\begin{lemma}[{\cite[Section 2, p.~108]{Cochran-Orr-Teichner:2004-1}}] \label{lemma:properties-rho-invariant}
Let $M$ be a closed 3-manifold and $\phi\colon \pi_1(M)\to G$ a homomorphism.
\begin{enumerate}
	\item $\rhot(-M,\phi) = -\rhot(M,\phi)$.
	\item \textup{(Subgroup property)} If $i\colon G\to G'$ is an injective homomorphism, then 
	\[
	\rhot(M,i\circ \phi) = \rhot(M,\phi).
	\]
	\item If $\phi$ is trivial, then $\rhot(M,\phi)=0$.
	\item If $\phi\colon \pi_1(M(K))\to \Z$ is a surjective homomorphism, then 
	\[
	\rhot(M(K),\phi)= \int_{S^1}\sigma_K(\omega)\,d\omega
	\] 
	where $S^1$ is the unit circle in $\C$ and the integral is normalized so that $\int_{S^1}d\omega = 1$. 
\end{enumerate}
\end{lemma}
 
 To estimate $\rhot(M,\phi)$, we use the following lemma.
 \begin{lemma}[{\cite{Cheeger-Gromov:1985-1}}]\label{lemma:universal-bound}
 Let $M$ be a closed 3-manifold. Then, there exists a constant $C_M$ such that $|\rhot(M,\phi)|\le C_M$ for all homomorphisms $\phi$.
 \end{lemma}
 
We remark that in Lemma~\ref{lemma:universal-bound}, if $M=M(K)$ for a knot $K$, then an explicit universal upper bound is given by $C_M=69713280\cdot c(K)$ where $c(K)$ is the crossing number of $K$ (see \cite[Theorem~1.9]{Cha:2016-1}).

\section{$H_1$--embedding of 3-manifolds in 4-manifolds}\label{section:embedding}
In this section, we recall groups that are amenable and in Sterebel's class $D(R)$, and provide a proof of Theorem~\ref{theorem:H1-embedding}.

A group $G$ is \emph{amenable} if it admits a finite additive measure that is invariant under the left multiplication. For a commutative ring $R$ with unity, a group $G$ is said to lie in \emph{Strebel's class $D(R)$} if a homomorphism $f\colon P\to Q$ between projective $RG$-modules is injective whenever the homomorphism $1_R\otimes_{RG}f\colon R\otimes_{RG}P\to R\otimes_{RG}Q$ is injective. (Here, we view $R$ as an $(R, RG)$-bimodule with trivial $G$-action.)

A large class of groups are amenable and in Strebel's class $D(R)$. For example, a group~$\Q$ is amenable and in Strebel's class $D(\Z)$, and finite $p$-groups are amenable and in Strebel's class $D(\Z_p)$. A group $G$ is said to be \emph{poly-torsion-free-abelian} if it admits a subnormal series of finite length whose successive quotient groups are torsion-free abelian, and every poly-torsion-free-abelian group is amenable and in Strebel's class $D(R)$ for any ring $R$. For more discussions about amenable groups lying in Strebel's class $D(R)$, we refer the reader to \cite[Section~6]{Cha-Orr:2009-01}.

Now we give a proof of Theorem~\ref{theorem:H1-embedding}.

\begin{proof}[Proof of Theorem~\ref{theorem:H1-embedding}]
Recall that for a space $X$, we define $b_i(X)=\rank_\Z H_i(X)$.  For a homomorphism $\Z[\pi_1(X)]\to \L$, let 
\[
b_i(X;\L) = \rank_\L H_i(X;\L).
\] 
For convenience, we identify $M$ with $f(M)$ via $f$. 

The complement $W\sm M$ has two connected components since both $M$ and $W$ are oriented and $f$ induces an isomorphism $H_1(M)\xrightarrow{\cong} H_1(W)$.
Let $W_1$ and $W_2$ be the closures of the connected components of $W\sm M$. Then, 
\[
\partial W_1=\partial W_2=M \text{ and }W_1\cup_M W_2=W,
\] 
and (1) follows.

Notice that $f_*\colon H_1(M)\to H_1(W)$ factors through $H_1(W_1)$ and $H_1(W_2)$, respectively. Property~(2) readily follows from (1), the Mayer--Vietoris sequence, and that $H_1(M)\cong \Z$ and $f_*\colon H_1(M)\to H_1(W)$ is an isomorphism.

From the Mayer--Vietoris sequence, we obtain the exact sequence
\[
H_2(W;\L)\to H_1(M;\L)\to H_1(W_1;\L)\oplus H_1(W_2;\L)\to H_1(W;\L).
\]
We will show that $H_2(W;\L)\cong \L^r$. One can readily see that $b_i(W)=1$ for $i=0,1,3,4$. Therefore, the Euler characteristic of $W$ equals $b_2(W)=r$, and hence
\begin{equation}\label{equation:euler-characteristic}
r=\sum_{i=0}^4 (-1)^ib_i(W;\L).
\end{equation}
We have $H_1(W;\L)=0$ since $\pi_1(W)\cong \Z$. It follows that $b_3(W;\L)=b_1(W;\L)=0$ from Poincar\'{e} duality and the universal coefficient theorem over the PID~$\L$. One can also see readily that  $b_0(W;\L)=b_4(W;\L)=0$. It follows from \eqref{equation:euler-characteristic} that $b_2(W;\L)=~r$. 
Furthermore, 
\[
H_2(W;\L)\cong H^2(W;\L)\cong \Hom_\L(H_2(W;\L), \L)
\]
since $H_1(W;\L)=0$. It follows that $H_2(W;\L)$ is a free module over $\L$. (Or, one can observe this from that $H_2(W;\Z[t^{\pm1 }])$ is $\Z[t^{\pm 1}]$--free since $\pi_1(W)\cong \Z$. For instance, see \cite[Lemma~3.2]{Conway-Powell:2023-1}.) Therefore, $H_2(W;\L)\cong \L^r$, and Property~(3) follows. 

Suppose that $G$ is an amenable group lying in Strebel's class $D(R)$ where $R=\Q$ or $\Z_p$ with $p$ a prime and that a homomorphism $\phi\colon \pi_1(M)\to G$ extends over $W_i$ for $i=1$ or 2. 

It follows readily from the definition of the $L^2$-signature $\sign_G^{(2)}(W_i)$ that 
\[
\left|\sign_G^{(2)}(W_i)\right| \le \dim^{(2)} H_2(W_i; \cN G),
\] 
where $\dim^{(2)}$ denotes the $L^2$-dimension (see \cite[Section~3.1]{Cha:2010-01}). We have 
\[
\dim^{(2)} H_2(W_i; \cN G)\le \dim_R H_2(W_i;R)
\]
by \cite[Theorem~3.11(1)]{Cha:2010-01}.
Also $\dim_R H_2(W_i;R) \le b_2(W_i)$ since $H_1(W_i)\cong H_1(M)\cong \Z$. Therefore, we obtain 
\[
\left|\sign_G^{(2)}(W_i)\right|\le b_2(W_i).
\] 
Obviously, $|\sign(W_i)|\le b_2(W_i)$, and it follows from \eqref{equation:rho-invariant} that $|\rhot(M,\phi)|\le 2b_2(W_i)$.

Therefore, $|\rhot(M,\phi)|\le 2r$ since $b_2(W_i)\le b_2(W)=r$, and Property~(4) follows.
\end{proof}

\section{The double slice genus and $H_1$--embeddings}\label{section:genus-embedding}
In this section, we give a proof of Theorem~\ref{theorem:genus-to-embedding}. 

\begin{proof}[Proof of Theorem~\ref{theorem:genus-to-embedding}]
There exist surfaces $F_i$ in $X_i$ with $\partial F_i=K$, $i=1,2$, such that $\text{genus}(F_1\cup F_2)\le g$ and $\pi_1((X_1\cup_{S^3} X_2)\sm (F_1\cup_K F_2))\cong \Z$ since $g_{X_1,X_2}(K)\le g$.
By taking the connected sum with an unknotted surface in the collar of $S^3$ in $X_1$ if necessary, we may assume that $\text{genus}(F_1\cup F_2)=g$.  

For $i=1,2$, let $\tilde{X_i}$ be a 4-manifold obtained by attaching a 2-handle to $X_i$ along the 0--framing of $K$ and let $\tilde{F_i}$ be a closed surface in $\tilde{X_i}$ obtained by taking a union of the core of the 2-handle and $F_i$ along $K$. 
Note that $\partial \tilde{X_i} = M(K)$, $H_1(\tilde{X_i})=0$, and $H_2(\tilde{X_i})\cong \Z^{b_2(X_i)+1}$. 
Recall that each $F_i$ is null-homologous in $X_i$ since $F_1\cup F_2$ is a $\Z$--surface and hence null-homologous. Therefore, the self-intersection of $\tilde{F_i}$ vanishes and $\tilde{F_i}$ has trivial normal bundle, $i=1,2$. 
There is a bijection between the set of framings on $\tilde{F_i}$ and $H^1(\tilde{\Sigma_i})$, and there is a framing on $\tilde{F_i}$ such that the induced homomorphism $H_1(\tilde{F_i})\to H_1(\tilde{X_i}\sm \tilde{F_i})$ is trivial (cf. \cite[p.~899]{Cha:2006-1}). 
We identify the tubular neighborhood of $\tilde{F_i}$ in $\tilde{X_i}$ with $\tilde{F_i}\times D^2$ using this framing. 

For $i=1,2$, let $Z_i = \tilde{X_i}\sm \text{int}(\tilde{F_i}\times D^2)$. One can readily show that $H_1(Z_i)\cong \Z$, which is generated by a meridian of $K$, and $b_2(Z_i)= b_2(X_i)+1+2g_i$.  
Let $H_i$ be a handlebody with boundary $\tilde{F_i}$, and define a 4-manifold $W_i$ as
\[
W_i = Z_i\cup \left({H_i\times S^1} \right),
\]
where $\tilde{F_i}\times S^1\subset \partial Z_i$ is identified with $\partial H_i\times S^1 = \partial (H_i\times D^2)$. Using the Mayer--Vietoris sequence and our choice of the framing on $\tilde{F_i}$, one can readily show that $H_1(W_i)\cong \Z$ and $b_2(W_i)=b_2(X_i)+2g_i$.
We also have $\partial W_i=M(K)$. 
Define a 4-manifold $W$ as
\[
W = W_1\cup_{M(K)} W_2.
\]

We will show that $\pi_1(W)\cong \Z$, $H_2(W)\cong \Z^{2g+b_2(X)}$, and the inclusion $M(K)\hookrightarrow W$ is an $H_1$--embedding.

Notice that $g=g_1+g_2$ and $b_2(X)=b_2(X_1)+b_2(X_2)$. Using the Mayer--Vietoris sequence of the pair $(W_1, W_2)$, one can readily see that 
\[
H_1(W)\cong \Z \text{ and } b_2(W)= 2g+b_2(X),
\]
Here, $H_1(W)$ is generated by a meridian of $K$, and therefore the inclusion $M(K) \hookrightarrow W$ is an $H_1$--embedding. The group $H_2(W)$ is free abelian since $H_1(W)\cong \Z$, and it follows that $H_2(W)\cong \Z^{2g+b_2(X)}$.

It remains to show that $\pi_1(W)\cong \Z$. For brevity, let
\[
Z = Z_1\cup_{M(K)} Z_2.
\]
Notice that $\pi_1(Z_i)\cong \pi_1(X_i\sm \Sigma_i)/\langle \ell\rangle$, $i=1,2$, and $\pi_1(M(K))\cong \pi_1(S^3\sm K)/\langle \ell\rangle$, where $\ell$ and $\langle -\rangle$ denote the preferred longitude of $K$ and the subgroup normally generated by $-$, respectively.  Using the Seifert--Van Kampen Theorem one can readily see that 
\[
\pi_1(Z)\cong \pi_1((X_1\cup_{S^3} X_2)\sm (F_1\cup_K F_2))/\langle \ell \rangle \cong \Z/\langle \ell\rangle.
\]
It follows from the Mayer--Vietoris sequence that $H_1(Z)\cong \Z$. We obtain that $\pi_1(Z)\cong \Z$ since $H_1(Z)\cong \Z$ and $\pi_1(Z)$ is a quotient of $\Z$. 
Notice that
\[
W=Z\cup \left(\bigsqcup_{i=1}^2 H_i\times S^1\right).
\]
Noticing that the homomorphism $\pi_1(\partial H_i\times S^1)\to \pi_1(H_i\times S^1)$ induced from inclusion is surjective for each $i$, using the Seifert--Van Kampen Theorem one can readily see that $\pi_1(W)$ is a quotient of $\pi_1(Z)\cong \Z$. 
We recall that $H_1(W)\cong \Z$, and therefore it follows that $\pi_1(W)\cong \Z$.
This proves Theorem~\ref{theorem:genus-to-embedding}~(1).

Theorem~\ref{theorem:genus-to-embedding}~(2) immediately follows from Theorem~\ref{theorem:genus-to-embedding}~(1) and the definition of~$\gdsx(K)$.
\end{proof}

Using Theorem~\ref{theorem:genus-to-embedding}, we give the following theorem, which gives a characterization for a knot to be doubly slice.
\begin{theorem}\label{theorem:doubly-slice-and-H_1-embedding}
	A knot $K$ is doubly slice if and only if there is an $H_1$--embedding of $M(K)$ in~$S^1\times S^3$.
\end{theorem}
\begin{proof}
Suppose that $K$ is doubly slice. Then, $g_{ds}(K)=0$, and by Theorem~\ref{theorem:genus-to-embedding}~(2), we obtain a closed 4-manifold $W$ with $\pi_1(W)\cong \Z$ and $H_2(W)=0$ such that there is an $H_1$--embedding of $M(K)$ in $W$. We will show that $W$ is homeomorphic to $S^1\times S^3$.

We have $H^2(W;\Z_2)\cong H_2(W;\Z_2)=0$ since $H_1(W)\cong \Z$ and $H_2(W)=0$, and therefore the 4-manifold $W$ is spin. We also have $H_1(W;\Z[t^{\pm 1}])=0$ since $\pi_1(W)\cong \Z$, and therefore $H_2(W;\Z[t^{\pm 1}])$ has no $\Z$-torsion. 
Furthermore, as in the proof of Theorem~\ref{theorem:H1-embedding}, one can show that $H_2(W;\L)=0$, and it follows that $H_2(W;\Z[t^{\pm 1}])=0$.

Therefore, $W$ is a spin closed 4-manifold with $\pi_1(W)\cong \Z$ that has trivial intersection form on $\pi_2(W)\cong H_2(W;\Z[t^{\pm 1}])=0$.
Now it follows from \cite[Theorem~F]{Kreck:1999-1} that $W$ is homeomorphic to $S^1\times S^3$ (also see \cite[Section~10.7]{Freedman-Quinn:1990-1}).	

Suppose that there is an $H_1$--embedding $f\colon M(K)\to S^1\times S^3$. For convenience, we identify~$M(K)$ with $f(M(K))$. 
The complement $S^1\times S^3 \sm M(K)$ has two components since the homomorphism $H_1(M(K))\to H_1(S^1\times S^3)$ induced from $f$ is an isomorphism. 
Let $V_1$ and~$V_2$ be the closures of the components of $S^1\times S^3\sm M(K)$. 
Then, $V_i$ are 4-manifolds with $\partial V_i=M(K)$, a meridian of $K$ normally generates $\pi_1(V_i)$, and $H_2(V_i)=0$, $i=1, 2$.

Recall that $M(K)=(S^3\sm \nu K)\cup_{S^1\times S^1}(S^1\times D^2)$ where $\nu K$ denotes the open tubular neighborhood of $K$ in $S^3$. 
Let $\mu$ be the curve $S^1\times 0\subset S^1\times D^2\subset M(K)$, which is a meridian of $K$. For $i=1, 2$, let $W_i$ be the 4-manifold obtained by attaching a 2-handle $D^2\times D^2$ to $V_i$ along $\mu$ with 0--framing. 
Then, $\partial W_i=S^3$, $\pi_1(W_i)=0$, and $H_2(W_i)=0$. It follows that $W_i$ is homeomorphic to $D^4$. 
Let~$D_i$ be the cocore of the 2-handle $0\times D^2$ in $W_i$, $i=1, 2$. 
Then, $\partial D_i$ represents the knot~$K$ in $S^3=\partial W_i$. 

Let $W=W_1\cup_{S^3} W_2$, which is homeomorphic to~$S^4$, and let $S=D_1\cup_K D_2$, which is homeomorphic to $S^2$. 
We show that $S$ is unknotted in $W$, which will imply that $K$ is doubly slice. 
Note that $W\sm S = (W_1\sm D_1)\cup_{S^3\sm K}(W_2\sm D_2)$.
It follows from the Seifert-Van Kampen Theorem that 
\begin{align*}
\pi_1(W\sm S)	&\cong \pi_1(W_1\sm D_1)*_{\pi_1(S^3\sm K)}\pi_1(W_2\sm D_2) \\
			&\cong \pi_1(V_1) *_{\pi_1(S^3\sm K)}\pi_1(V_2) \\
			&\cong \pi_1(V_1) *_{\pi_1(M(K))}\pi_1(V_2) \\
			&\cong \pi_1(S^1\times S^3) \\
			&\cong\Z.
\end{align*}
It follows from \cite[Theorem~11.7A]{Freedman-Quinn:1990-1} that $S$ is unknotted in $W\cong S^4$. 
\end{proof}

\section{The $X$--double slice genus of a knot}\label{section:double-silce-genus-in-4-manifolds}
In this section, we present knots with nontrivial $X$-double slice genus and give a proof of Theorem~\ref{theorem:main-ribbon}.

\begin{proof}[Proof of Theorem~\ref{theorem:main-ribbon}]
Let $R$ be the knot $9_{46}$, which is ribbon. Let $\eta$ be a simple closed curve that goes around the left band of the obvious Seifert surface for $R$ once and is unknotted in~$S^3$. See Figure~\ref{figure:knot-R}.

\begin{figure}[htb!]
	\includegraphics{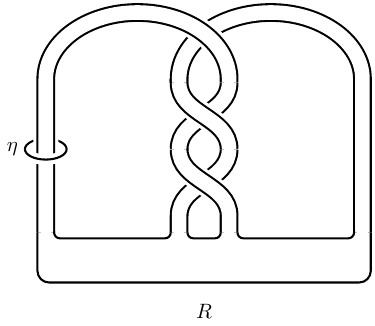}
	\caption{The knot $R$}\label{figure:knot-R}
\end{figure}

Recall that $\L=\Q[t^{\pm 1}]$. We choose an orientation of $R$ so that 
\[
H_1(M(R);\L)\cong \L/\langle 2t-1\rangle \oplus \L/\langle t-2\rangle
\]
and $\eta$ generates the submodule $\L/\langle 2t-1\rangle$. 

For a knot $K$, let $\rho_0(K)$ denote $\rhot(M(K),\epsilon)$ where $\epsilon$ is the abelianization homomorphism $\pi_1(M(K))\to \Z$. Let $J$ be a knot such that $\rho_0(J)>4g+2n$.
For instance, we may choose~$J$ to be the connected sum of $N$ copies of the left-handed trefoil $J_0$ where $N$ is an integer greater than $3g+\frac{3}{2}n$, noticing that $\int_{S^1}\sigma_{J_0}(\omega)\,d\omega=4/3$ (see Lemma~\ref{lemma:properties-rho-invariant}~(4)). 

For $1\le i \le 2g+n+1$, let $\eta_i=\eta$ and let $K_i=R(\eta_i, J)$, a satellite knot with pattern~$R$, companion~$J$, and axis~$\eta_i$. Note that each $K_i$ is a ribbon knot since there is a band surgery on the left band of the obvious Seifert surface for $K_i$ that produces a 2-component unlink.
Also $K_i$ is algebraically doubly slice since the linking number between $\eta_i$ and $R$ is zero, the knot $K_i$ and $R$ share the same Seifert forms, and the Seifert form associated with the obvious Seifert surface of $R$ is hyperbolic.  

Let $K=\#_{i=1}^{2g+n+1} K_i$. Then, the knot $K$ is algebraically doubly slice and ribbon. 

Suppose that there exists a closed 4-manifold $W$ such that $\pi_1(W)\cong \Z$, $b_2(W)\le 2g+n$ and there exists an $H_1$--embedding of~$M(K)$ in~$W$. By taking the connected sum of $W$ and copies of $\C\P^2$ if necessary, we may assume that $b_2(W)=2g+n$. Then, there exist 4-manifolds $W_1$ and~$W_2$ satisfying the properties in Theorem~\ref{theorem:H1-embedding}, with $r=2g+n$.

Let $A_{2t-1}$ and $A_{t-2}$ be the $(2t-1)$--primary part and the $(t-2)$--primary part of $H_1(M(K);\L)$, respectively. Note that 
\[
H_1(M(K);\L)\cong \left(\L/\langle 2t-1\rangle\right)^{2g+n+1} \oplus \left(\L/\langle t-2\rangle\right)^{2g+n+1} = A_{2t-1}\oplus A_{t-2},
\]
and $A_{2t-1}$ is generated by $\{\eta_1,\eta_2, \ldots, \eta_{2g+1}\}$.

Note that the sequence
\[
\L^{2g+n}\cong H_2(W;\L)\xrightarrow{\partial_*} H_1(M(K);\L)\xrightarrow{((f_1)_*,(f_2)_*)} H_1(W_1;\L)\oplus H_1(W_2;\L)\to 0
\]
is exact by Theorem~\ref{theorem:H1-embedding}~(3) and $A_{2t-1}$ is not generated by $2g+n$ elements. It follows that the image of~$\partial_*$ does not contain $A_{2t-1}$. Therefore, there exists some $\eta_i\in H_1(M(K);\L)$ that is not contained in the image of $\partial_*$. By rearranging $K_i$ if necessary, we may assume $\eta_1\notin \Im \partial_*$. Then, either $(f_1)_*(\eta_1)$ or $(f_2)_*(\eta_1)$ is nontrivial. By exchanging $W_1$ with $W_2$ if necessary, we may assume that $(f_1)_*(\eta_1)$ is nontrivial.

Let $A'_{2t-1}$ be the $(2t-1)$--primary part of $H_1(W_1;\L)$. For a group $G$, we define inductively the \emph{$n$th derived subgroup} of $G$ for each integer $n\ge 0$ by letting $G^{(0)}=G$ and $G^{(n+1)}=[G^{(n)},G^{(n)}]$ for $n\ge 0$. Note that there is an isomorphism
\begin{align*}
\pi_1(W_1)/\pi_1(W_1)^{(2)}	&\cong \left(\pi_1(W_1)^{(1)}/\pi_1(W_1)^{(2)}\right)\rtimes \left(\pi_1(W_1)/\pi_1(W_1)^{(1)}\right)\\
& \cong H_1(W_1;\Z[t^{\pm 1}])\rtimes \langle t\rangle.
\end{align*}

Let $G= A'_{2t-1}\rtimes \langle t\rangle$, and let $\phi\colon \pi_1(M(K))\to G$ be the composition
\begin{multline*}
\pi_1(M(K))\to \pi_1(W_1)\to \pi_1(W_1)/\pi_1(W_1)^{(2)} \to \\
H_1(W_1;\Z[t^{\pm 1}])\rtimes \langle t\rangle \to H_1(W_1;\L)\rtimes \langle t\rangle \to G, 
\end{multline*}
where the last map is induced by the projection $H_1(W_1;\L)\to A'_{2t-1}$.

For convenience and brevity, we will write $\phi$ for restrictions of $\phi$ to subspaces of $W_1$. Notice that the group $G$ is poly-torsion-free-abelian since $A'_{2t-1}$ and $\langle t\rangle$ are torsion-free abelian. Therefore, $G$ is amenable and in Strebel's class $D(\Z)$, and by Theorem~\ref{theorem:H1-embedding}~(4) we obtain
\[
|\rhot(M(K),\phi)|\le 4g+2n.
\]
We will show that $\rhot(M(K),\phi)> 4g+2n$ using a different computation, which will lead to a contradiction and complete the proof.

Note that $\phi\colon \pi_1(M(K))\to G$ can be decomposed as $\phi=h\circ \psi$, where $\psi$ is the composition
\[
\psi\colon \pi_1(M(K))\to \pi_1(M(K))/\pi_1(M(K))^{(2)}\to H_1(M(K);\L)\rtimes \langle t\rangle
\] 
and
\[
h\colon H_1(M(K);\L)\rtimes \langle t\rangle \to H_1(W_1;\L)\rtimes \langle t\rangle \to G
\]
is defined by $h=((f_1)_*,\text{id})$.

Let $C$ be the standard cobordism with $\partial C=\partial_+ C\sqcup \partial_-C$ where
	\[
	\partial_+ C = \bigsqcup_{i=1}^{2g+n+1}M(K_i) \quad\mbox{ and }\quad \partial_- C= -M(K)
	\]
(see \cite[p.~113]{Cochran-Orr-Teichner:2004-1}). One can see that $\pi_1(C)\cong \pi_1(M(K))/N$ where $N$ is the normal subgroup generated by the 0--framed longitudes $\ell_i$ of $K_i$, $1\le i\le 2g+n+1$. 

Since $\psi(\ell_i)$ is trivial for each $i$, so is $\phi(\ell_i)$, and  it follows that $\phi\colon \pi_1(M(K))\to G$ extends to $\pi_1(C)\to G$, which we also denote by $\phi$. 
For clarity, we denote by $\phi_i$ the restriction of $\phi\colon \pi_1(C)\to G$ to $\pi_1(M(K_i))$, $1\le i\le 2g+n+1$. 
From the construction of $C$, one can readily see that for any group~$\Gamma$, the map $H_2(\partial C;\Z\Gamma)\to H_2(C;\Z\Gamma)$ is surjective and hence $\sign_\Gamma^{(2)}(C)=0$ (see \cite[Lemma~2.4]{Cochran-Harvey-Leidy:2009-1} and its proof). 
In particular, 
\[
\sign_G^{(2)}(C)=\sign(C)=0.
\]

Therefore, it follows that $\rhot(\partial C,\phi) = \sign_G^{(2)}(C)-\sign(C)=0$, and we have
\[
 \rhot(M(K),\phi) = \sum_{i=1}^{2g+n+1}\rhot(M(K_i),\phi_i).
\] 
Let $\nu R$ denote the open tubular neighborhood of $R$ in $S^3$.
Since $\phi_i(\ell_i)$ is trivial for each $i$, the map $\phi_i$ on $S^3\sm \nu R\subset M(K_i)$ extends to $M(R)$, and we denote it by $\phi_i'$. 

By \cite[Lemma~2.3]{Cochran-Harvey-Leidy:2009-1}, 
\[
\rhot(M(K_i),\phi_i) = \rhot(M(R),\phi_i') + a_i\rho_0(J),
\]
where $a_i=1$ if $\phi_i'(\eta_i)$ is nontrivial, and $a_i=0$ otherwise. 

Since $(f_1)_*(\eta_1)$ is nontrivial in $H_1(W_1;\L)$ and it is $(2t-1)$-torsion, it follows that~$\phi_1'(\eta_1)$ is nontrivial in $G$ and $a_1=1$. Therefore,
\[
\sum_{i=1}^{2g+n+1}\left(\rhot(M(R),\phi_i') + a_i\rho_0(J)\right) 
\ge \rho_0(J) + \sum_{i=1}^{2g+n+1}\rhot(M(R),\phi_i').
\]
One can see that $\phi_i'\colon \pi_1(M(R))\to G$ extends over the complement of a ribbon disk for $R$ in $D^4$ which is obtained by cutting the right band of the obvious Seifert surface for $R$ (note that a curve dual to the right band maps to zero under $\phi_i'$). 
Therefore, $\rhot(M(R),\phi_i')=0$ for all $i$. It follows that
\[
\rhot(M(K),\phi) = \sum_{i=1}^{2g+n+1}\left(\rhot(M(R),\phi_i') + a_i\rho_0(J)\right) \ge \rho_0(J) >4g+2n,
\]
which is a contradiction.

The last part of Theorem~\ref{theorem:main-ribbon} that $\gdsx(K)>g$ follows from Theorem~\ref{theorem:genus-to-embedding}~(2).
\end{proof}

We finish this section with the following theorem, which shows the relationship between $\gdsx(K)$ and $\gds^{S^4}(K)$.
\begin{theorem}\label{theorem:inequalities-double-slice-genus}
Let $K$ be a knot in $S^3$ and $X$ be a simply connected closed 4-manifold. Then,
\[
\gdsx(K)\le \gds^{S^4}(K).
\]
\end{theorem}
\begin{proof}
For brevity, let $g=\gds^{S^4}(K)$. Let $D_1$ and $D_2$ be copies of $D^4$. 
Then, there exist properly embedded surfaces $F_i\subset D_i$ with $\partial F_i=K$, $i=1, 2$, such that $F_1\cup_K F_2$ is a $\Z$--surface in $D_1\cup_{S^3} D_2=S^4$ and $\text{genus}(F_1\cup_K F_2)=g$. 

Let $X_1$ and $X_2$ be 4-manifolds with boundary $S^3$ such that $X_1\cup_{S^3} X_2=X$ and $g_{X_1, X_2}(K)=\gdsx(K)$.
We can embed $F_i$ into the collar of $S^3$ in $X_i$, $i=1, 2$, and now $F_1\cup_K F_2$ is embedded in the bicollar of $S^3$, say $B$, in $X$. 
Moreover, $\pi_1(B\sm (F_1\cup_K F_2))\cong \Z$ since it is a $\Z$--surface in~$S^4$. 
One can readily show that $\pi_1(X\sm (F_1\cup_K F_2))\cong \Z$ using the Seifert-Van Kampen Theorem since all $X_i$ are simply connected.

Therefore, 
\begin{equation*}
\gdsx(K)=g_{X_1, X_2}(K)\le \text{genus}(F_1\cup_K F_2) = g=\gds^{S^4}(K). \qedhere
\end{equation*}
\end{proof}

\section{The double slice genus and knots whose prime power fold branched cyclic covers embed in $S^4$}\label{section:vanishing-CG-invariants}

In this section, we give a proof of Theorem~\ref{theorem:main-vanishing-CG}. 

\begin{proof}[Proof of Theorem~\ref{theorem:main-vanishing-CG}]
For $n\ge 1$, let $\Phi_n(t)$ denote the $n$th cyclotomic polynomial. 
Let $R$ be a knot with $H_1(M(R);\L)\cong \L/\langle\Phi_{30}(t)\rangle$ and let $\eta$ be a simple closed curve in $\pi_1(S^3\sm R)^{(1)}$ that generates $H_1(M(R);\L)$ and is unknotted in $S^3$. 
By Lemma~\ref{lemma:universal-bound}, there exists a constant~$D$ such that $|\rhot(M(R\# (-R)),\phi)|\le D$ for all $\phi$. 

Recall that $\rho_0(J)$ denotes $\rhot(M(J),\epsilon)$ where $\epsilon\colon \pi_1(M(J))\to \Z$ is the abelianization. 
Let~$J$ be a knot such that 
\[
\rho_0(J)> 4g+(2g+1)D.
\] 
For a choice of $J$, one may take a connected sum of sufficiently many copies of the left-handed trefoil (see Lemma~\ref{lemma:properties-rho-invariant}). 

For $1\le i\le 2g+1$, let $\eta_i=\eta$ and let 
\[
K_i=R(\eta_i, J)\# (-R) = (R\# (-R))(\eta_i,J),
\]
a satellite knot with companion~$J$, pattern~$R\#(-R)$, and axis~$\eta_i$. Now let 
\[
K=\#_{i=1}^{2g+1} K_i.
\] 
We will show that $K$ has the desired properties.

Since $\eta\in \pi_1(S^3\sm R)^{(1)}$, each $K_i$ has the same Seifert form as $R\#(-R)$. Since $R\#(-R)$ is (algebraically) doubly slice, it follows that each $K_i$ is algebraically doubly slice, and so is $K$. 

The only irreducible factor of the Alexander polynomial of $K$ is $\Phi_{30}(t)$, and therefore $H_1(\Sigma_n(K))=0$ for every prime power $n$ by \cite[Theorem~1.2]{Livingston:2002-1}. Therefore, Property ~(2) follows.

Suppose that there exists an $H_1$--embedding of $M(K)$ in a closed 4-manifold $W$ with $\pi_1(W)\cong~\Z$ and $b_2(W)\le 2g$. By taking the connected sum with copies of $\C P^2$ if necessary, we may assume that $b_2(W)=2g$. 

Then, there exist 4-manifolds $W_1$ and $W_2$ which satisfy the properties in Theorem~\ref{theorem:H1-embedding}. By Theorem~\ref{theorem:H1-embedding}~(3), we have an exact sequence
\[
\L^{2g}\cong H_2(W;\L)\xrightarrow{\partial_*} H_1(M(K);\L)\xrightarrow{((f_1)_*,(f_2)_*)} H_1(W_1;\L)\oplus H_1(W_2;\L)\to 0.
\]
Since 
\[
H_1(M(K);\L)\cong \bigoplus^{4g+2} \L/\langle \Phi_{30}(t)\rangle
\]
and the submodule generated by $\eta_1,\ldots, \eta_{2g+1}$ cannot be generated by $2g$ elements, there exists some $\eta_i$ that is not contained in the image of $\partial _*$. Rearranging $K_i$ if necessary, we may assume that $\eta_1$ is not contained in the image of $\partial_*$. Therefore, either $(f_1)_*(\eta_1)$ or $(f_2)_*(\eta_1)$ is nontrivial.
We may assume that $(f_1)_*(\eta_1)$ is nontrivial by exchanging $W_1$ and $W_2$ with each other if necessary. 

Let $G=H_1(W_1;\L)\rtimes \langle t\rangle$, and let $\phi\colon \pi_1(W_1)\to G$ be the composition
\[
\pi_1(W_1)\to \pi_1(W_1)/\pi_1(W_1)^{(2)}\cong H_1(W_1;\Z[t^{\pm 1}])\rtimes \langle t\rangle \to G.
\]
By an abuse of notation, we write $\phi$ for restrictions of $\phi$ to subspaces of $W_1$. 

Since $\partial W_1=M(K)$, we have 
\[
\rhot(M(K),\phi)=\sign_G^{(2)}(W_1)-\sign(W_1).
\] 
Note that $G$ is poly-torsion-free-abelian. Therefore, $G$ is amenable and in Strebel's class $D(\Z)$. It follows that $|\rhot(M(K),\phi)|\le 4g$ by Theorem~\ref{theorem:H1-embedding}~(4). We will show that $\rhot(M(K),\phi)> 4g$ using a different computation, which will lead to a contradiction and complete the proof.

Using an argument similar to that in the proof of Theorem~\ref{theorem:main-ribbon}, we obtain that 
\[
 \rhot(M(K),\phi) = \sum_{i=1}^{2g+1}\rhot(M(K_i),\phi_i).
\] 
Since $\phi_i(\ell_i)$ is trivial for each $i$, the map $\phi_i$ on $S^3\sm \nu (R\# (-R))\subset M(K_i)$ extends to $M(R\#(-R))$, and we denote it by $\phi_i'$. 

By \cite[Lemma~2.3]{Cochran-Harvey-Leidy:2009-1}, 
\[
\rhot(M(K_i),\phi_i) = \rhot(M(R\#(-R)),\phi_i') + a_i\rho_0(J),
\]
where $a_i=1$ if $\phi_i'(\eta_i)$ is nontrivial, and $a_i=0$ otherwise. 

Note that $\phi_i'(\eta_i) = \phi_i(\eta_i) = ((f_1)_*(\eta_i),0)$ and it follows that $\phi_1'(\eta_1)$ is nontrivial. Therefore, 
\begin{align*}
\sum_{i=1}^{2g+1}\rhot(M(K_i),\phi_i) &= \sum_{i=1}^{2g+1}\left(\rhot(M(R\#(-R)),\phi_i') + a_i\rho_0(J)\right)\\
&\ge \rho_0(J) + \sum_{i=1}^{2g+1}\rhot(M(R\#(-R)),\phi_i') \\
&> (4g+(2g+1)D) - (2g+1)D\\
&= 4g.
\end{align*}
Therefore, $\rhot(M(K),\phi)> 4g$, which is a contradiction.	

One can readily see that $\gds^{S^4}(K)> g$ using Theorem~\ref{theorem:genus-to-embedding}~(2) with $X=S^4$. It follows from $\gds(K)\ge \gds^{S^4}(K)$ that $\gds(K)> g$. 
\end{proof}

\section{The double stabilizing number of a knot}
\label{section:doubly-stabilizing-number}

In this section, we investigate the double stabilizing number of a knot and prove Theorem~\ref{theorem:main-dsn}.

\subsection{The double stabilizing number}
\label{subsection:double-stabilizing-number}

\begin{theorem}\label{theorem:stably-doubly-slice-Arf-invariant}
Let $K$ be a knot in $S^3$.
\begin{enumerate}
	\item $K$ is stably doubly slice if and only if it is stably $H$--slice.
	\item Suppose that $K$ has vanishing Arf invariant. Then,
	\[
	2\sn(K)\le \dsn(K)\le  \gds(K).
	\]
\end{enumerate}
\end{theorem}
\begin{proof}
It is obvious that a stably doubly slice knot is stably $H$--slice. Suppose that $K$ is stably $H$--slice. 
Let $g=\gds(K)$ for brevity. We will show that $K$ is $(D^4\# mS^2\times S^2, D^4\# nS^2\times S^2)$--doubly slice for some integers $m$ and $n$ such that $g=m+n$, and hence $K$ is stably doubly slice. 

Let $F$ be a closed $\Z$--surface in $S^4$ such that $K=S^3\cap F$ and $\text{genus}(F)=g$. 
We have $S^4=D_1^4\cup_{S^3} D_2^4$, where $D_i^4\cong D^4$, $i=1,2$. 
Let $F_i=F\cap D_i^4$ and let $g_i$ be the genus of $F_i$, $i=1, 2$. Then, $\partial F_1=\partial F_2=K$ and $g=g_1+g_2$. 

Recall that $K$ is stably $H$--slice if and only if $K$ has vanishing Arf invariant \cite{Schneiderman:2010-1}. 
Since~$K$ has vanishing Arf invariant, it follows from the proof of \cite[Theorem~5.15]{Conway-Nagel:2020-1} that one can perform surgery along curves $\gamma_1^i, \ldots, \gamma_{g_i}^i$ on $F_i$ in $D_i^4$ and obtain properly embedded null-homologous disks $D_i$ in $D_i^4\# g_iS^2\times S^2$ such that $\partial D_i=K$, $i=1,2$. 
Let $W_i=D_i^4\# g_iS^2\times S^2$, $i=1, 2$. 
Noticing $\partial W_1=\partial W_2=S^3$, let $W=W_1\cup_{S^3}W_2$ and  let $S=D_1\cup_K D_2$. Then, $W\cong S^4\# gS^2\times S^2$, $S\cong S^2$, and $S\cap S^3=K$.

Furthermore, $\pi_1(W\sm S)\cong \Z$, which can be seen from the following observation: we have $\pi_1(S^4\sm F)\cong \Z$, and $W$ and $S$ are obtained by performing surgeries on $S^4$ and $F$, respectively, along a collection of curves $\{\gamma_j^i\mid 1\le j\le g_i\text{ and }i=1,2\}$ on~$F$.
The surgery on~$S^4$ is performed by taking off $\gamma_j^i\times \text{int}(D^3)$ and then taking a union with $D^2\times S^2$ along the common boundary $\gamma_j^i\times S^2$. 
While doing this, $F$ is surgered to~$S$ by taking off $\gamma_j^i\times \text{int}(D^1)$ from $F$ and then taking a union with $D^2\times S^0\subset D^2\times S^2$ along the common boundary $\gamma_j^i\times S^0\subset \gamma_j^i\times S^2$. 
Let $\nu F$, $\nu S$, and $\nu S^0$ denote the open tubular neighborhoods of $F$, $S$, and $S^0$ in $S^4$, $W$, and $S^2$, respectively.
It follows that $W\sm \nu S$ is a union of $S^4\sm \nu F$ and $g D^2\times (S^2\sm \nu S^0)$ along $\gamma^i_j\times (S^2\sm \nu S^0)$. 
Noticing that $[\gamma^i_j\times \{*\}]=0\in \pi_1(S^4\sm F)$ and a generating curve of $\pi_1(S^2\sm \nu S^0)\cong \Z$ is homotopic to a meridian of $S^4\sm F$, one can readily see that $\pi_1(W\sm S)\cong \Z$ using the Seifert--Van Kampen Theorem. 

Therefore, $K$ is $(W_1,W_2)$--doubly slice. This completes the proof of~(1). 
One can readily see that the above proof also shows that $\dsn(K)\le g=\gds(K)$. 

We complete the proof of (2) by showing that $2\sn(K)\le \dsn(K)$. 
Let $n=\dsn(K)$. Then,~$K$ is $(X_1, X_2)$--doubly slice where $X_i=D^4\# n_iS^2\times S^2$ for some integers $n_i$ such that $n=n_1+n_2$. 
Therefore, there exist disks $D_i$ in $X_i$ with $\partial D_i=K$ such that $D_1\cup D_2$ is a $\Z$--sphere in $X_1\cup X_2$. The disks $D_i$ are null-homologous in $X_i$ since $D_1\cup D_2$ is a $\Z$--sphere and hence null-homologous in $X_1\cup X_2$ (see \cite[Lemma~5.1]{Conway-Powell:2023-1}).
It follows that
\[
2\sn(K)\le 2\cdot\text{min}\{n_1,n_2\}\le n_1+n_2=n=\dsn(K),
\]
and this proves~(2).
\end{proof}

The following theorem give the relationship between the double stabilizing number and $H_1$--embeddings. 
\begin{theorem}\label{theorem:H_1-embedding-and-dsn}
Let $m$ be a nonnegative integer and $K$ be a stably doubly slice knot. If $\dsn(K)\le m$, then there exists a closed 4-manifold $W$ with $\pi_1(W)\cong \Z$ and $H_2(W)\cong \Z^{2m}$ such that there is an $H_1$--embedding of $M(K)$ in $W$. 
\end{theorem}
\begin{proof}
We may assume that $\dsn(K)=m$. Then, $K$ is $(X_1, X_2)$--doubly slice where $X_i=D^4\# m_iS^2\times S^2$ for some integers $m_i$, $i=1,2$, such that $m=m_1+m_2$. Notice that $\text{genus}(D_1\cup D_2)=0$, $H_1(X_1\cup X_2)=0$, and $b_2(X_1\cup X_2)=2m$. Now the theorem follows from Theorem~\ref{theorem:genus-to-embedding}~(1) with $g=0$ and $n=2m$.
\end{proof}

Using Theorem~\ref{theorem:H_1-embedding-and-dsn}, we give a proof of Theorem~\ref{theorem:main-dsn}.
\begin{proof}[Proof of Theorem~\ref{theorem:main-dsn}]
By Theorem~\ref{theorem:main-ribbon}, there exists an algebraically doubly slice and ribbon knot~$K$ such that  there is no $H_1$--embedding of $M(K)$ in a closed 4-manifold~$W$ with $\pi_1(W)\cong~\Z$ and $H_2(W)\cong \Z^{2m}$. It follows from Theorem~\ref{theorem:H_1-embedding-and-dsn} that $\dsn(K)>~m$.
\end{proof}

\subsection{The $X$--double stabilizing number}
\label{subsection:X-double-stabilizing-number}

For a simply connected closed 4-manifold $X$, we extend the double stabilizing number $\dsn(K)$ to \emph{the $X$--double stabilizing number}, which we denote by $\dsn_X(K)$, as follows. For 4-manifolds~$X_1$ and~$X_2$ with boundary $S^3$ and  an integer $n\ge 0$, let $X_{i,n} = X_i\# nS^2\times S^2$, $i=1, 2$. 

\begin{definition}\label{X-double-stabilizing-number}
Let $X_1$ and $X_2$ be simply connected 4-manifolds with boundary $S^3$ and $K$ be a knot in $S^3$. 
\begin{enumerate}
	\item The \emph{$(X_1,X_2)$--double stabilizing number of $K$} is 
\[
\dsn_{X_1, X_2}(K) = \min\{m+n\mid K \text{ is } (X_{1,m}, X_{2,n})\text{--doubly slice} \}.
\]
	\item Let $X$ be a simply connected closed 4-manifold. The \emph{$X$--double stabilizing number of~$K$} is
\begin{align*}
\dsn_X(K)=\text{min} \{\dsn_{X_1, X_2}(K) \mid\, & X_i\text{ are 4-manifolds with boundary }S^3, i=1,2,\\
&  \text{such that } X=X_1\cup_{S^3}X_2\}.
\end{align*}
\end{enumerate}
\end{definition}

We note that $\dsn(K) = \dsn_{S^4}(K)$. The following theorem extends Theorems~\ref{theorem:H_1-embedding-and-dsn} and \ref{theorem:main-dsn}.

\begin{theorem}\label{theorem:H_1-embedding-and-X-dsn}
Let $K$ be a knot in $S^3$ and $m$ be a nonnegative integer. Let $X$ be a simply connected closed 4-manifold.

\begin{enumerate}
	\item If $\dsn_X(K)\le m$, then there exists a closed 4-manifold $W$ with $\pi_1(W)\cong \Z$ and $H_2(W)\cong \Z^{b_2(X)+2m}$ such that there is an $H_1$--embedding of $M(K)$ in $W$. 
	\item For each integer $m\ge 0$, there exists an algebraically doubly slice and ribbon knot $K$ such that $\dsn_X(K)> m$. 
\end{enumerate}
\end{theorem}
\begin{proof}
The proof is similar to the proofs of Theorems~\ref{theorem:H_1-embedding-and-dsn} and \ref{theorem:main-dsn}. The statement~(1) follows from Theorem~\ref{theorem:genus-to-embedding}~(1) with $g=0$ and $n=b_2(X)+2m$. The statement~(2) follows from Theorem~~\ref{theorem:main-ribbon} and the statement~(1).
\end{proof}

\section{The $X$--superslice genus of a knot}
\label{section:superslice-genus}

In this section, we prove Theorems~\ref{theorem:lower-bound-superslice-genus} and \ref{theorem:superslice-genus-large}.

\begin{proof}[Proof of Theorem~\ref{theorem:lower-bound-superslice-genus}]
For brevity, let $g=g_s^X(K)$.
There is a properly embedded null-homologous surface~$F$ in~$X$ with $\partial F=K$ such that $\text{genus}(F) = g$ and~$DF$ is a $\Z$--surface in~$DX$. 
Then, it follows from the proof of Theorem~\ref{theorem:genus-to-embedding} that there exist 4-manifolds~$W_i$, $i=1,2$, with boundary $M(K)$ such that $W_1=W_2$ and, for $W=W_1\cup_{M(K)} W_2$, we have $\pi_1(W)\cong \Z$, $H_2(W)\cong \Z^{4g+2b_2(X)}$ and the inclusion $M(K)\hookrightarrow W$ is an  $H_1$--embedding. 

Then, there is an exact sequence by Theorem~\ref{theorem:H1-embedding} and its proof such that
\[
\L^{4g+2b_2(X)}\to H_1(M(K);\L)\xrightarrow{((f_1)_*,(f_2)_*)} H_1(W_1;\L)\oplus H_1(W_2;\L)\to 0.
\]
Note that $W_1=W_2$, $(f_1)_*=(f_2)_*$, and the map $((f_1)_*,(f_2)_*)$ is surjective in the above exact sequence. Therefore one can conclude that $H_1(W_i;\L)=0$, $i=1, 2$. It follows that we have an exact sequence 
\[
\L^{4g+2n}\to H_1(M(K);\L)\to 0,
\]
and the theorem follows.
\end{proof}

\begin{proof}[Proof of Theorem~\ref{theorem:superslice-genus-large}]
Let $J$ be any doubly slice and ribbon knot with nontrivial Alexander polynomial such that $H_1(M(J);\L)$ is cylic. For instance, one can choose $R=9_{46}$ for $J$.
Let~$K$ be the connected sum of $4g+2b_2(X)+1$ copies of $J$. Then, $K$ is doubly slice, ribbon, and the minimal number of generators of $H_1(M(K);\L)$ is $4g+2b_2(X)+1$. It follows from Theorem~\ref{theorem:lower-bound-superslice-genus} that
\[
4g+2b_2(X)+1 \le 4g_s^X(K)+2b_2(X),
\]
and therefore $g_s^X(K)>g$.

There exists an unknotted $2$--sphere $S$ in the bicollar, say $B$, of~$S^3$ in~$DX$ such that $S\cap S^3=K$ since the knot~$K$ is doubly slice. In particular, $\pi_1(B\sm S)\cong\Z$. Since $X$ is simply connected, using the Seifert--Van Kampen Theorem one can readily see that $\pi_1(DX\sm S)\cong \Z$ and therefore $\gds^{DX}(K)=0$.
\end{proof}

We finish the section with the following theorem. 
\begin{theorem}\label{theorem:inequalities-superslice-genus}
Let $K$ be a knot in $S^3$ and $X$ be a simply connected 4-manifold with boundary~$S^3$. Then,
$g_s^X(K)\le g_s^{D^4}(K)$.
\end{theorem}
\begin{proof}
The proof is similar as that of Theorem~\ref{theorem:inequalities-double-slice-genus}.
Let $F$ be a properly embedded surface of genus $g_s^{D^4}(K)$ in $D^4$ such that $\partial F=K$ and $DF$ is a $\Z$--surface in the double of $D^4$. 
Let $B$ be a collar of $S^3$ in $X$ and we can embed $F\subset B$. Then, $DF$ is a $\Z$--surface in $DB$. 
One can readily show that $DF$ is a $\Z$--surface in $DX$ using the Seifert-Van Kampen Theorem since $X$ is simply connected. Therefore,
\begin{equation*}
g_s^X(K)\le \text{genus}(F) = g_s^{D^4}(K). \qedhere
\end{equation*}
\end{proof}

\bibliographystyle{amsalpha}
\def\MR#1{}
\bibliography{research}

\end{document}